\documentclass{amsart}

\usepackage{amsmath,amssymb,amsthm}
\usepackage{enumerate}
\usepackage{todonotes}
\usepackage[unicode=true]{hyperref}
\usepackage{url}

\numberwithin{equation}{section}
\theoremstyle{plain}
\newtheorem{thm}{Theorem}[section]
\newtheorem{lem}[thm]{Lemma}
\newtheorem{prop}[thm]{Proposition}
\theoremstyle{definition}
\newtheorem{defn}[thm]{Definition}
\title{Inertia of Retracts in Demushkin groups}
\author{Henrique Souza}
\date{November 4, 2021}

\begin{document}

\begin{abstract} Exploring inequalities regarding the rank and relation gradients of pro-$p$ modules and building upon recent results of Y. Antolín, A. Jaikin-Zapiran and M. Shusterman, we prove that every retract of a Demushkin group is inert in the sense of the Dicks-Ventura Inertia Conjecture.
\end{abstract}

\maketitle

\section{Introduction}

In 1996, W. Dicks and E. Ventura introduced the concept of inertia for subgroups of the abstract free groups of finite rank $F_n$: a subgroup $H$ of $F_n$ is inert if $\operatorname{rk}(H\cap K) \leq \operatorname{rk}(K)$ for every subgroup $K$ of $F_n$ (\cite{dicksventura}). In that same monograph, they proved the subgroup of fixed points of any family of injective endomorphisms of $F_n$ is inert, and conjectured whether this holds for the subgroup of fixed points of any family of endomorphisms. It can be shown (\cite[Thm. 1]{turner}, \cite[Conj. 8.1]{venturasurvey}) that such subgroups are inert in some retract of $F_n$, and thus the now called Dicks-Ventura Inertia Conjecture is equivalent to retracts of $F_n$ being inert. This conjecture and it's version for surface groups were proven in \cite[Cor. 1.5]{antolin_zapirain_20} by Y. Antolín and A. Jaikin-Zapirain.

This paper concerns the problem of inertia for the retracts of Demushkin groups, defined in Section~\ref{sec-preliminaries}. For a general pro-$p$ groups, we define inertia and retracts as follows:

The rank of a finitely generated pro-$p$ group $G$ is the cardinality of a minimal set of topological generators of $G$, and this cardinal number is denoted by $d(G)$. A closed subgroup $H$ of a pro-$p$ group $G$ is inert if $d(H\cap K) \leq d(K)$ for every closed subgroup $K$ of $G$. Since this inequality holds for all subgroups of $K$ if $K$ is not finitely generated, it suffices to check inertia for the finitely generated such $K$.

We say that a closed subgroup $H$ of a pro-$p$ group $G$ is a retract of $G$ if there exists a homomorpshim $\tau\colon G \to H$ that extends the identity map on $H$. Equivalently, $H$ is a retract of $G$ if there exists a closed normal subgroup $N$ of $G$ such that $G \simeq N \rtimes H$. If $H$ is a retract of $G$, then $H$ is also a retract of every closed subgroup $K$ of $G$ containing $H$, in which case we have $d(H) \leq d(K)$. Moreover, equality is achieved if and only if $H = K$, since $H/\Phi(H)$ is a direct factor of $K/\Phi(K)$.

In the free pro-$p$ case, the inertia of retracts can be proven by more elementary means. If $H$ is a retract of a finitely generated free pro-$p$ group $F$, then $H$ is a free factor of $F$ (\cite[Lem. 3.1]{lubotzkyCombinatorialGroupTheory1982}). If $K$ is any finitely generated subgroup of $F$, then $H\cap K$ is a free factor of $K$ by the pro-$p$ analogue of Kurosh's subgroup theorem  (\cite[Thm. 4.4]{herfortSubgroupsFreePropproducts1987}) and therefore $d(H\cap K) \leq d(K)$.

Our main result can then be stated as:

\begin{thm}\label{thm_retracts_are_inert} Every retract of a Demushkin group is inert.
\end{thm}

The paper is organized as follows: in Section~\ref{sec-preliminaries} we define and state some properties of Demushkin groups and it's homological gradients, and the proof of Theorem~\ref{thm_retracts_are_inert} is given in Section~\ref{sec_inertia}. We also fix some notations: for any pro-$p$ group $G$, we denote by $[\![\mathbb{F}_p G]\!]$ the complete group algebra of $G$ over the field $\mathbb{F}_p$ of $p$ elements. In general, if $X$ is a profinite space, $[\![\mathbb{F}_p X]\!]$ denotes the free pro-$p$ $\mathbb{F}_p$-module (vector space) over $X$. The complete tensor product of two pro-$p$ $[\![\mathbb{F}_p G]\!]$-modules $A$ and $B$ is denoted by $A \hat{\otimes}_{[\![\mathbb{F}_p G]\!]} B$. For any closed subgroup $H$ of a pro-$p$ group $G$ and any $[\![\mathbb{F}_p H]\!]$-module $M$, $\operatorname{Ind}^G_H M$ denotes the induced $[\![\mathbb{F}_p G]\!]$-module $[\![\mathbb{F}_p G]\!] \hat{\otimes}_{[\![\mathbb{F}_p H]\!]} M$.

\subsection*{Acknowledgements} The author would like to thank Andrei Jaikin-Zapirain, Pavel Zalesskii and Theo Zapata for helpful comments. He also thanks Andrei Jaikin-Zapirain for presenting him to the problem of inertia of retracts. The contents of this paper form part of the author's M.Sc. dissertation at the University of Brasilia, presented under the advice of Theo Zapata.

\section{Preliminaries}\label{sec-preliminaries}

We say that a pro-$p$ group $G$ is a Demushkin group if $G$ is a pro-$p$ Poincaré duality group in dimension $2$, that is, if $G$ satifies the three conditions below:
\begin{enumerate}
	\item $H^i(G,\mathbb{F}_p)$ is finite for each $i \geq 0$;
	\item $\dim_{\mathbb{F}_p} H^2(G, \mathbb{F}_p) = 1$;
	\item The cup product $$\cup\colon H^i(G,\mathbb{F}_p) \times H^{2-i}(G,\mathbb{F}_p) \to H^2(G,\mathbb{F}_p) $$ is a non-degerate bilinear form for every $i \geq 0$.
\end{enumerate}
Some examples of Demushkin groups are the cyclic group of order two $\mathbb{Z}/2\mathbb{Z}$, semidirect products of the form $\mathbb{Z}_p \rtimes \mathbb{Z}_p$ and the pro-$p$ completions of orientable surface groups. The first one is the only finite Demushkin group and the only one to have infinite cohomological dimension, and the first two comprise the class of all solvable Demushkin groups.

We recall the rank formula for open subgroups of Demushkin groups: if $U$ is an open subgroup of a Demushkin group $G$, then $U$ is also a Demushkin group and its rank satisfies
\begin{equation}\label{eq_demushkin_rank_formula}
    d(U) = (G\colon U)(d(G) - 2) + 2\,.
\end{equation}
Moreover, if $H$ is a closed subgroup of $G$ with infinite index, then $H$ is a free pro-$p$ group (\cite[Section~4.5]{serre_97}).

We now define some homological gradients of Demushkin groups that will be used on Section~\ref{sec_inertia}.

For a pro-$p$ group $G$ and a profinite $[\![\mathbb{F}_p G]\!]$-module $M$, we say that $M$ is finitely generated if there exists a finite collection of elements $m_1,\ldots,m_k \in M$ such that every element of $M$ can be written as a finite linear combination of the $m_i$ with coefficients in $[\![\mathbb{F}_p G]$. We say that $M$ is finitely related as an $[\![\mathbb{F}_p G]\!]$-module if $H_1(G, M)$ is finite. If $M$ is both finitely generated and finitely related, we say that $M$ is a finitely presented $[\![\mathbb{F}_p G]$-module.

By Nakayama's lemma, we know that $M$ is a finitely generated $[\![\mathbb{F}_p G]\!]$-module if and only if $$M/([\![\mathbb{F}_p G]\!] - 1)M \simeq M_G \simeq H_0(G,M)$$ is finite. If $M$ is finitely generated and $$0 \to N \to F \to M \to 0$$ is a presentation of $M$ with $F$ a free $[\![\mathbb{F}_p G]\!]$-module of minimal rank, the long exact sequence induced in homology gives us an isomorphism $$H_1(G,M) \simeq H_0(G,N)\,,$$ from which we deduce that $M$ is finitely related if and only if $N$ is a finitely generated $[\![\mathbb{F}_p G]\!]$-module.

\begin{defn} For a pro-$p$ group $G$ and a profinite $[\![\mathbb{F}_p G]\!]$-module $M$, the rank gradient $\beta_0^G(M)$ and the relation gradient $\beta_1^G(M)$ are defined as
the infimum $$\beta_i^G(M) = \inf_{U \unlhd_o G} \frac{\dim_{\mathbb{F}_p} H_i(U,M)}{(G\colon U)}\,.$$
\end{defn}

The gradients are always non-negative (and possibly infinite) real numbers. Since we have $$\frac{\dim_{\mathbb{F}_p} H_i(U,M)}{(G\colon U)} \leq \frac{\dim_{\mathbb{F}_p} H_i(V,M)}{(G\colon V)}$$ whenever $U \leq V$ by \cite[Lem. 4.2]{jaikin_shusterman_19}, we see that if $M$ is finitely generated (resp. presented), then $\beta_0^G(M)$ (resp. $\beta_1^G(M)$) is finite. We recall some properties of the rank and relation gradients:

\begin{prop}[{\cite[Section~4]{jaikin_shusterman_19}}]\label{prop_properties_of_gradients} Let $G$ be a finitely presented pro-$p$ group. Then, for any $i \in \{0,1\}$ and any $[\![\mathbb{F}_p G]\!]$-module $M$ such that $\beta_i^G(M)$ is finite, the following statements hold:
\begin{enumerate}[(i)]
    \item For any open subgroup $U$ of $G$, we have $\beta_i^U(M) = (G\colon U)\beta_i^G(M)$;
    \item For any closed subgroup $H$ of $G$ and any $[\![\mathbb{F}_p H]\!]$-module $N$ such that $\beta_i^H(N)$ is finite, we have $\beta_i^G(\operatorname{Ind}^G_H N) = \beta_i^H(N)$;
    \item If $M'$ is a submodule of $M$ and $M'' = M/M'$, then $$\beta_i^G(M) \leq \beta_i^G(M') + \beta_i^G(M'')\,;$$
\end{enumerate}
\end{prop}

For a closed and finitely generated subgroup $H$ of an infinite Demushkin group $G$, the relation gradient of the induced module $\operatorname{Ind}^G_H \mathbb{F}_p \simeq [\![\mathbb{F}_p (G/H)]\!]$ satisfies
\begin{equation}\label{eq_relation_gradient_demushkin}
    \beta_1^G([\![\mathbb{F}_p(G/H)]\!]) =
    \begin{cases}
        d(H) - 2\,,\ \text{ if }(G\colon H) < \infty\,,\\
        d(H) - 1\,,\ \text{ otherwise.}
    \end{cases}
\end{equation}
This is a consequence of the rank formula~(\ref{eq_demushkin_rank_formula}) when $(G\colon H)$ is finite and of Schreier's formula (\cite[Thm.~3.6.2]{ribes_zalesskii_10}) when $(G\colon H)$ is infinite. We also need the following inequality for the relation gradients of $[\![\mathbb{F}_pG]\!]$-modules when $G$ is an infinite Demushkin group:

\begin{lem}[{\cite[Prop. 4.6]{jaikin_shusterman_19}}]\label{lem_criteria_inequality_relation_gradient} Let $G$ be a non-solvable Demushkin group and $M$ be a finitely related $[\![\mathbb{F}_pG]\!]$-module. If $N$ is a submodule of $M$ such that $M/N$ is either finite or satisfies $H_2(G,M/N) = 0$, then $\beta_1^G(N) \leq \beta_1^G(M)$.
\end{lem}

\section{Inertia of retracts}\label{sec_inertia}

For any pro-$p$ group $G$ and any finitely generated closed subgroup $H$ of $G$, define the $[\![\mathbb{F}_pG]\!]$-module $I_{G/H}$ as the kernel of the map $[\![\mathbb{F}_p(G/H)]\!] \to \mathbb{F}_p$ that sends every coset $gH \in G/H$ to the multiplicative identity $1$ in $\mathbb{F}_p$. We write simply $I_G$ for $I_{G/\{1\}}$.

Suppose that $G$ is a Demushkin group. From the long exact sequence on cohomology associated with the inclusion $I_{G/H} \to [\![\mathbb{F}_p(G/H)]\!]$ we find that $$\dim_{\mathbb{F}_p} H_1(G,I_{G/H}) < \dim_{\mathbb{F}_p} H_1(G,[\![\mathbb{F}_p(G/H)]\!]) + \dim_{\mathbb{F}_p} H_2(G,\mathbb{F}_p) = d(H) + 1\,,$$ from which we conclude that $\beta_1^G(I_{G/H})$ is always finite.

\begin{lem}\label{lem_independent_rank} Let $G$ be an infinite Demushkin group and $H$ a closed subgroup of $G$ such that $\beta_1^G(I_{G/H}) = 0$. Then, $d(H) \leq d(G)$. Moreover, if $G$ is solvable, then every closed subgroup $H$ of $G$ is such that $\beta_1^G(I_{G/H}) = 0$. If $G$ is not solvable and the inclusion $H \subseteq G$ is proper, then the index $(G\colon H)$ is infinite, that is, $H$ is a free pro-$p$ group.
\end{lem}
\begin{proof} The first part is a straightforward computation:
\begin{align*}
    d(H) - 2 &\leq \beta_1^G([\![\mathbb{F}_p(G/H)]\!])\,,\ \text{ by }(\ref{eq_relation_gradient_demushkin})\\
    &\leq \beta_1^G(I_{G/H}) + \beta_1^G(\mathbb{F}_p) = d(G) - 2\,.
\end{align*}

If $G$ is solvable, then for every closed subgroup $H$ of $G$ we have that $$\beta_1^G([\![\mathbb{F}_p(G/H)]\!]) = 0$$ by~\ref{eq_relation_gradient_demushkin}, since $d(H)$ equals $2$ or $1$ according to the index $(G\colon H)$. Thus, $\beta_1^G(I_{G/H}) = 0$ by Lemma~\ref{lem_criteria_inequality_relation_gradient}. Suppose now that $G$ is not solvable, that is, $d(G) > 2$. If the index $(G\colon H)$ is finite, then after substituting $d(H)$ in the inequality $d(H) \leq d(G)$ by the expression on the rank formula of the equation~(\ref{eq_demushkin_rank_formula}), we find that $(G\colon H) = 1$.
\end{proof}

The following proposition is a small extension of \cite[Prop. 4.5]{antolin_zapirain_20}, and it provides us with our main source of subgroups $H$ of Demushkin groups $G$ with vanishing $\beta_1^G(I_{G/H})$.

\begin{prop}\label{prop_retracts_are_independent} Any proper retract $H$ of an infinite Demushkin group $G$ is a finitely generated free pro-$p$ subgroup of $G$ satisfying $\beta_1^G(I_{G/H}) = 0$.
\end{prop}
\begin{proof} Observe that the induced map on the Frattini quotients $H/\Phi(H) \to G/\Phi(G)$ is injective, and therefore the vanishing of $\beta_1^G(I_{G/H})$ follows from \cite[Prop. 4.5]{antolin_zapirain_20}. Let $\tau\colon G \to H$ be the retraction map from $G$ onto a proper retract $H$. Since the index $(G\colon H)$ equals the order of the kernel $\ker(\tau)$ and $G$ is torsion-free, we find that $(G\colon H) = \infty$ and therefore $H$ is a free pro-$p$ group which is finitely generated (because it is a quotient group of $G$).
\end{proof}

Finally, let $H$ be a proper retract of an infinite Demushkin group $G$ and take any closed and finitely generated subgroup $K$ of $G$. By Proposition~\ref{prop_retracts_are_independent}, we have $\beta_1^G(I_{G/H}) = 0$. Since all Demushkin groups possess Howson's property \cite[Thm.~3.1]{shusterman_zalesskii_20}, the intersection $H\cap K$ is a finitely generated subgroup, and therefore the relation gradient $\beta_1^K(I_{K/(H\cap K)})$ is also defined. Hence, Theorem~\ref{thm_retracts_are_inert} is implied by the vanishing of $\beta_1^K(I_{K/(H\cap K)})$, which we shall now prove.

\begin{thm}\label{thm_intersection_independent_subgroup} Let $H$ be a subgroup of a Demushkin group $G$ with $\beta_1^G(I_{G/H}) = 0$, and let $K$ be a closed and finitely generated subgroup of $G$. Then, $\beta_1^K(I_{K/(H\cap K)}) = 0$.
\end{thm}
\begin{proof} Without loss of generality, we may assume that $H$ is a proper subgroup of $G$. If $G$ is solvable, then the claim follows from Lemma~\ref{lem_independent_rank}, so henceforth we also assume that $d(G) > 2$. Thus, $H$ is a finitely generated free pro-$p$ group by Lemma~\ref{lem_independent_rank}.

Let $P$ be the kernel of the surjection $[\![\mathbb{F}_pK]\!] \to [\![\mathbb{F}_p(K/H\cap K)]\!]$. We have an isomorphism of $[\![\mathbb{F}_pK]\!]$-modules $I_{K/(H\cap K)} \simeq I_K/P$. If $O$ denotes the kernel of $[\![\mathbb{F}_p G]\!] \to [\![\mathbb{F}_p (G/H)]\!]$, then after identifying $[\![\mathbb{F}_p K]\!]$ with a submodule of $[\![\mathbb{F}_p G]\!]$ we also find that $P = I_K\cap O$. Hence, we also have an isomorphism of $K$-modules $I_{K/(H\cap K)} \simeq I_K/(I_K\cap O)$. This allows us to identify $I_{K/(H\cap K)}$ with an $[\![\mathbb{F}_p K]\!]$-submodule of $I_G/O \simeq I_{G/H}$, since $P$, $O$ and $I_K$ are $[\![\mathbb{F}_p K]\!]$-submodules of $I_G$. By choosing a complete set of representatives $X$ for the double cosets $H\backslash G/K$ with $1 \in X$, we obtain the following isomorphism for the quotient module:
\begin{equation} \label{eq_isomorphism_relative_augmentation_kernels}
\begin{aligned}
	I_{G/H}/I_{K/(H\cap K)} &\simeq (I_G/O)/(I_{K}/I_K \cap O)\\
	&\simeq [\![\mathbb{F}_p(G/H)]\!]/[\![\mathbb{F}_p(K/H\cap K)]\!]\\
	&\simeq \left(\bigoplus_{x \in X} [\![\mathbb{F}_p(K/xHx^{-1}\cap K)]\!]\right)/[\![\mathbb{F}_p(K/H\cap K)]\!]\,.
\end{aligned}
\end{equation}

Since $K$ may have infinite index, the set $X$ may be infinite. Thus, the direct sum of pro-$p$ modules appearing in~(\ref{eq_isomorphism_relative_augmentation_kernels}) denotes the direct sum of modules indexed by a profinite space as in \cite[Section 9.1]{ribes_17}. We shall freely use the fact that this direct sum commutes with the homology operator (\cite[Thm. 9.1.3]{ribes_17}), that is, for all $i \in \mathbb{N}$ we have an isomorphism:
$$H_i\left(K,\bigoplus_{x \in X} [\![\mathbb{F}_p(K/xHx^{-1}\cap K)]\!]\right) \simeq \bigoplus_{x \in X} H_i(K, [\![\mathbb{F}_p(K/xHx^{-1}\cap K)]\!])\,.$$

Suppose at first that $K$ is open in $G$, in which case the isomorphism~(\ref{eq_isomorphism_relative_augmentation_kernels}) becomes $$I_{G/H}/I_{K/(K\cap H)} \simeq \bigoplus_{x \in X - \{1\}} [\![\mathbb{F}_p(K/xHx^{-1}\cap K)]\!]\,.$$ Then, since $\beta_1^G(I_{G/H}) = 0$, we find that $$\beta_1^K(I_{G/H}) = (G\colon K)\beta_1^G(I_{G/H}) = 0\,.$$ Therefore, by Lemma~\ref{lem_criteria_inequality_relation_gradient}, it suffices to prove that $$H_2(K,I_{G/H}/I_{K/(H\cap K)}) \simeq \bigoplus_{x \in X-\{1\}} H_2(K,[\![\mathbb{F}_p(K/xHx^{-1}\cap K)]\!]) = 0\,.$$ However, $xHx^{-1} \cap K$ is a free pro-$p$ group, and thus we apply Shapiro's lemma to obtain $$H_2(K,[\![\mathbb{F}_p(K/xHx^{-1}\cap K)]\!]) \simeq H_2(xHx^{-1}\cap K, \mathbb{F}_p) = 0\,.$$ Hence, $\beta_1^K(I_{K/(H\cap K)}) = 0$.

Suppose now that the index $(G\colon K)$ is infinite and we shall prove again that $\beta_1^K(I_{G/H})$ vanishes. Since $$\beta_1^K(I_{G/H}) = \beta_1^G(\operatorname{Ind}^G_K I_{G/H}) = \beta_1^G([\![\mathbb{F}_p (K\backslash G)]\!] \hat{\otimes}_{\mathbb{F}_p} I_{G/H})\,,$$ it suffices to show the inequality
\begin{equation}\label{eq_submultiplicativity_independence}
    \beta_1^G([\![\mathbb{F}_p (K\backslash G)]\!] \hat{\otimes}_{\mathbb{F}_p} I_{G/H}) \leq \beta_1^G([\![\mathbb{F}_p (K\backslash G)]\!]) \beta_1^G(I_{G/H}) = 0\,.
\end{equation}
We procceed as in \cite[Section~8]{jaikin_shusterman_19}.

By \cite[Cor. 7.4]{jaikin_shusterman_19}, there is an open subgroup $U$ of $G$ and an $[\![\mathbb{F}_p U]\!]$-submodule $A$ of $[\![\mathbb{F}_p(K\backslash G)]\!]$ such that $\beta_1^U(A) = 0$ and
\begin{equation}\label{eq_first_thm_dimension_submodule_A}
    \dim_{\mathbb{F}_p} [\![\mathbb{F}_p(K\backslash G)]\!]/A \leq \beta_1^G([\![\mathbb{F}_p(K\backslash G)]\!])\,.
\end{equation}
Hence, we obtain the inequality
\begin{equation}\label{eq_first_thm_first_bound}
    \beta_1^U([\![\mathbb{F}_p(K\backslash G)]\!]\hat{\otimes}_{\mathbb{F}_p} I_{G/H}) \leq \beta_1^U(A\hat{\otimes}_{\mathbb{F}_p} I_{G/H}) + \beta_1^U(([\![\mathbb{F}_p(K\backslash G)]\!]/A)\hat{\otimes}_{\mathbb{F}_p} I_{G/H})\,.
\end{equation}
By repeatedly applying the subadditivity of the relation gradient and the inequality~(\ref{eq_first_thm_dimension_submodule_A}), we find that the last term of~(\ref{eq_first_thm_first_bound}) satisfies the bound
\begin{equation}\label{eq_first_thm_last_term}
    \beta_1^U(([\![\mathbb{F}_p(K\backslash G)]\!]/A)\hat{\otimes}_{\mathbb{F}_p} I_{G/H}) \leq \beta_1^G([\![\mathbb{F}_p(K\backslash G)]\!])\beta_1^U(I_{G/H}) = 0\,.
\end{equation}

The $[\![\mathbb{F}_p K]\!]$-module $I_{G/H}$ is finitely related since $K$ is finitely generated and there is an isomorphism $$H_1(K, I_{G/H}) \leq H_1(K,[\![\mathbb{F}_p(G/H)]\!]) \simeq \bigoplus_{x \in X} H_1(xHx^{-1} \cap K,\mathbb{F}_p)$$ with the latter group being finite by \cite[Cor. 3.4]{jaikin_shusterman_19} and Howson's property. Thus, we can apply \cite[Lem. 7.1]{jaikin_shusterman_19} to find an open $[\![\mathbb{F}_p G]\!]$-submodule $B$ of $I_{G/H}$ such that $\beta_1^K(B) = 0$. Once again, we have the inequality
\begin{equation}\label{eq_first_thm_second_bound}
    \beta_1^U(A \hat{\otimes}_{\mathbb{F}_p} I_{G/H}) \leq \beta_1^U(A \hat{\otimes}_{\mathbb{F}_p} B) + \beta_1^U(A \hat{\otimes}_{\mathbb{F}_p} (I_{G/H}/B))\,.
\end{equation}
As it were the case with~(\ref{eq_first_thm_last_term}), repeated applications of the subadditive property gives us the bound
\begin{equation}\label{eq_first_thm_first_term}
     \beta_1^U(A \hat{\otimes}_{\mathbb{F}_p} (I_{G/H}/B)) \leq \beta_1^U(A)\cdot \dim_{\mathbb{F}_p} (I_{G/H}/B) = 0\,.
\end{equation}

Hence, by combining the inequalities~(\ref{eq_first_thm_first_bound}) through~(\ref{eq_first_thm_first_term}), we obtain
\begin{equation}\label{eq_first_thm_last_bound}
     \beta_1^U([\![\mathbb{F}_p(K\backslash G)]\!]\hat{\otimes}_{\mathbb{F}_p} I_{G/H}) \leq \beta_1^U(A \hat{\otimes}_{\mathbb{F}_p} B)\,.
\end{equation}
Since $U$ has cohomological dimension 2, we have the inequalities
\begin{align*}
    \dim_{\mathbb{F}_p} H_2(U,B) &\leq \dim_{\mathbb{F}_p} H_2(U,I_{G/H})\\
    &\leq \dim_{\mathbb{F}_p} H_2(U,[\![\mathbb{F}_p(G/H)]\!])\\
    &= \dim_{\mathbb{F}_p} H_2(G, [\mathbb{F}_p(U/G)] \otimes_{\mathbb{F}_p} [\![\mathbb{F}_p (G/H)]\!])\text{, by Shapiro's lemma}\\
    &= \dim_{\mathbb{F}_p} H_2(H, [\mathbb{F}_p(U/G)])\text{, again by Shapiro's lemma}\\
    &= 0\,.
\end{align*}
Since $\dim_{\mathbb{F}_p} H_2(U,-)$ is also subadditive with respect to short exact sequences, we deduce that
\begin{equation}\label{eq_first_thm_H2_vanishes}
    H_2(U, ([\![\mathbb{F}_p(K\backslash G)]\!]/A)\hat{\otimes}_{\mathbb{F}_p} B) = 0\,.
\end{equation}

We wrap up the argument with the observation that Lemma~\ref{lem_criteria_inequality_relation_gradient} applies to the short exact sequence of $[\![\mathbb{F}_pU]\!]$-modules $$0 \to A\hat{\otimes}_{\mathbb{F}_p} B \to [\![\mathbb{F}_p(K\backslash G)]\!] \hat{\otimes}_{\mathbb{F}_p} B \to ([\![\mathbb{F}_p(K\backslash G)]\!]/A)\hat{\otimes}_{\mathbb{F}_p} B \to 0$$ by~(\ref{eq_first_thm_H2_vanishes}). Thus:
\begin{align}\label{eq_first_thm_last_vanishing}
    \beta_1^U(A\hat{\otimes}_{\mathbb{F}_p} B) &\leq \beta_1^U([\![\mathbb{F}_p(K\backslash G)]\!] \hat{\otimes}_{\mathbb{F}_p} B)\\
    &= (G\colon U)\beta_1^G(\operatorname{Ind}^G_K B)\notag\\
    &= (G\colon U)\beta_1^K(B) = 0\,.\notag
\end{align}
Multiplying by $(G\colon U)$ and using the index proportionality of the relation gradient (statement (i) of Proposition~\ref{prop_properties_of_gradients}), we see that the combination of~(\ref{eq_first_thm_first_bound}) through~(\ref{eq_first_thm_last_vanishing}) is readily equivalent to~(\ref{eq_submultiplicativity_independence}):
\begin{align*}
    \beta_1^G([\![\mathbb{F}_p(K\backslash G)]\!] \hat{\otimes}_{\mathbb{F}_p} I_{G/H}) &\leq (G\colon U)\beta_1^U([\![\mathbb{F}_p(K\backslash G)]\!] \hat{\otimes}_{\mathbb{F}_p} I_{G/H})\\
    &\leq (G\colon U)\beta_1^U(A\hat{\otimes}_{\mathbb{F}_p} I_{G/H})\,,\text{ by (\ref{eq_first_thm_first_bound}) and (\ref{eq_first_thm_last_term})}\\
    &\leq (G\colon U)\beta_1^U(A \hat{\otimes}_{\mathbb{F}_p} B)\,,\text{ by (\ref{eq_first_thm_second_bound}) and (\ref{eq_first_thm_first_term})}\\
    &= 0\,,\text{ by (\ref{eq_first_thm_last_vanishing}).}
\end{align*}
Since $K$ has cohomological dimension 1, there is the bound $$\beta_1^K(I_{K/(H\cap K)}) \leq \beta_1^K(I_{G/H}) = 0\,,$$ which shows that $\beta_1^K(I_{K/(H\cap K)}) = 0$.
\end{proof}

\begin{proof}[Proof of \ref{thm_retracts_are_inert}] If $G$ is solvable, then every closed subgroup $H$ of $G$ is inert; this follows from the fact that $d(K) \leq 2$ for every closed subgroup $K$ of $G$. If $G$ is not solvable, then the statement follows from Lemma~\ref{lem_independent_rank}, Proposition~\ref{prop_retracts_are_independent} and Theorem~\ref{thm_intersection_independent_subgroup}.
\end{proof}

\bibliographystyle{amsplain}
\bibliography{retracts}
\end{document}